\documentclass{article}

\usepackage{amsfonts,amsmath,amssymb,amsthm}
\usepackage{booktabs}  

\newtheorem{theorem}{Theorem}
\newtheorem{lemma}{Lemma}
\newtheorem{proposition}{Proposition}

\title{Effective dimension of some weighted pre-Sobolev spaces
with dominating mixed partial derivatives}
\author{Art B. Owen\\Stanford University}
\date{December 2018}

\usepackage{amsopn}

\newcommand{\natu}{\mathbb{N}}

\newcommand{\real}{\mathbb{R}}

\newcommand{\mrd}{\mathrm{d}}
\newcommand{\rd}{\,\mathrm{d}}

\renewcommand{\le}{\leqslant}
\renewcommand{\ge}{\geqslant}
\renewcommand{\emptyset}{\varnothing}

\newcommand{\bsa}{{\boldsymbol{a}}}

\newcommand{\bsx}{\boldsymbol{x}}
\newcommand{\bsy}{\boldsymbol{y}}
\newcommand{\bsz}{\boldsymbol{z}}

\newcommand{\bsgamma}{{\boldsymbol{\gamma}}}

\newcommand{\bszero}{\boldsymbol{0}}

\newcommand{\ca}{\mathcal{A}}
\newcommand{\cb}{\mathcal{B}}
\newcommand{\cf}{\mathcal{F}}

\newcommand{\phz}{\phantom{0}}
\newcommand{\hk}{\mathrm{HK}}
\newcommand{\e}{\mathbb{E}}

\begin{document}

\maketitle

\begin{abstract}
This paper considers two notions of effective dimension for  
quadrature in weighted pre-Sobolev spaces
with dominating mixed partial derivatives.
We begin by finding a ball in those spaces just barely
large enough to contain a function with unit variance.
If no function in that ball has more than $\varepsilon$
of its variance from ANOVA components involving 
interactions of order $s$ or more, then the space
has effective dimension at most $s$ in the superposition sense.
A similar truncation sense notion replaces the cardinality
of the ANOVA component by the largest index it contains.
Some Poincar\'e type inequalities  
are used to bound variance components by multiples  
of these space's squared norm and those  
in turn provide bounds on effective dimension.  
Very low effective dimension in the superposition  
sense holds for some spaces defined by product weights  
in which quadrature is strongly tractable. 
The superposition dimension is 
$O( \log(1/\varepsilon)/\log(\log(1/\varepsilon)))$
just like the superposition dimension used in the multidimensional
decomposition method.  Surprisingly, 
even spaces where all subset weights are equal, regardless  
of their cardinality or included indices, have low superposition  
dimension in this sense. 
This paper does not require periodicity of the integrands.
\end{abstract}


\smallskip
\par\noindent
{\bf AMS Categories:}
41A55, 65D30, 45E35

\section{Introduction}

This paper gives upper bounds for the 
effective dimension of certain weighted 
pre-Sobolev spaces of functions on $[0,1]^d$.
They are pre-Sobolev spaces because the requisite derivatives must exist
as continuous functions and not just in the distributional sense.
The notion of effective dimension used here is the
one in the technical report \cite{effdim-tr}.  That work was based on a Fourier
expansion and it required a periodicity condition
on certain partial derivatives of the functions $f$.  
In this paper, no periodicity assumption is required.

The weighted spaces we consider are used as models
for certain numerical integration problems.  
There we seek $\mu =\int_{[0,1]^d}f(\bsx)\rd\bsx$
where $1\le d<\infty$, and $d$
might be large.  For large $d$, integration becomes quite
hard for worst case integrands, even smooth ones, yielding
a curse of dimensionality described by Bakhvalov \cite{bakh:1959}.
Quasi-Monte Carlo (QMC) integration \cite{nied92,dick:pill:2010},
often succeeds in high dimensions despite the curse.
This can be explained by the functions having
less than full $d$-dimensional complexity.
Weighted Sobolev spaces \cite{hick:1996,sloa:wozn:1998}
provide one model for such reduced complexity.  
This paper translates the weights defining those spaces into bounds 
on certain $L^2$ norms quantifying the notion that those spaces 
of integrands are `effectively' $s$ dimensional where $s$ might be much less than $d$.

The functions we consider have a partial derivative
taken once with respect to each of $d$ coordinates,
and moreover, that partial derivative is a continuous function
on $[0,1]^d$.  That is sufficient smoothness to
place them in the weighted Sobolev spaces mentioned above.
Such weighted spaces have been used to model
settings in which higher order interactions 
\cite{hick:1996}, or successive dimensions, 
or both \cite{sloa:wozn:1998} are less
important. When these high order or high index components
decay quickly enough, the result is a 
set of functions that evades the curse of dimensionality
established by \cite{bakh:1959}.

The contribution of various high order or high index parts of
an integrand to the Monte Carlo (MC) variance can be quantified through
the analysis of variance (ANOVA) decomposition defined below.
Suppose that  an integrand $f$ has less than one percent of its variance
coming from high dimensional components, and that
another quadrature method proves to be far more accurate than plain MC.
That improvement cannot be 
attributed to better handling of the high dimensional parts,
because they caused at most one percent of the squared error for Monte Carlo.
The improvement must have come from superior handing of
the low dimensional aspects of $f$.

In this paper we investigate some senses in which
a whole space of functions is of low effective dimension.
We look at some weighted Sobolev norms 
(and some semi-norms) and define measures 
of the extent to which balls in those normed spaces 
are dominated by their low dimensional parts with respect 
to an ANOVA decomposition. 
We select a ball in which the worst case Monte Carlo
variance is unity, and then consider whether {\sl any}
integrand in that ball has meaningfully large variance
coming from its high dimensional components.
If not, then the space itself is said to have low effective dimension.

An outline of the paper is as follows.
Section~\ref{sec:notation} gives our notation, introducing
weighted pre-Sobolev spaces with some conditions on the weights
and some decompositions (ANOVA and anchored) of $L^2\bigl([0,1]^d\bigr)$.
Section~\ref{sec:literature} presents a survey of the
notion of effective dimension, dating back to 1951.
Section~\ref{sec:poincare} introduces some Poincar\'e type inequalities
that we use to lower bound the Sobolev squared norms in terms of
ANOVA components.
Section~\ref{sec:effdim} defines what it means for a ball
in a space of functions to have a given effective dimension, i.e., to lack
meaningfully large higher dimensional structure in any of
its functions.
Section~\ref{sec:bounds} gives upper bounds on the
effective dimension of a space in terms of its weights
under  monotonicity conditions that give smaller
weights to higher order and higher indexed subsets of
variables. Explicit effective dimension bounds are worked out
and tabulated.  
One surprise in this work is that giving every subset 
equal weight, regardless of its cardinality or the size 
of the indices it contains, still leads to modest superposition dimension. 
Section~\ref{sec:conclusions} has some conclusions.

To finish this section, we mention some related work
in addition to that covered in Section~\ref{sec:literature}.
Poincar\'e inequalities have been used in the global 
sensitivity analysis (GSA) literature to bound Sobol' indices.  See 
for instance \cite{sobo:kuch:2009,sobo:kuch:2010}
and \cite{lamb:ioos:pope:gamb:2012}. 
Roustant et al.\ \cite{rous:frut:ioos:kunh:2014} use Poincar\'e inequalities 
based on mixed partial derivatives to bound some 
superset variable importance measures in GSA. 

\section{Notation}\label{sec:notation}

The indices of $\bsx\in[0,1]^d$ are $j\in\{1,2,\dots,d\}\equiv 1{:}d$. 
For  $u\subseteq 1{:}d$, we use $|u|$ for its cardinality 
and $-u$ for its complement with respect to $1{:}d$. 
For $u\subseteq1{:}d$, the point $\bsx_u\in[0,1]^{|u|}$
consists of the components $x_j$ for $j\in u$. 
For $i=1,\dots,n$, the point $\bsx_{i,u}\in[0,1]^{|u|}$ has those components
from the point $\bsx_i\in[0,1]^d$.
The hybrid point $\bsy=\bsx_u{:}\bsz_{-u}$ 
has $y_j=x_j$ for $j\in u$ and $y_j=z_j$ for $j\not\in u$. 

The differential $\mrd\bsx_u$ is $\prod_{j\in u}\mrd x_j$. 
Similarly, $\partial^u f$ denotes $\partial^{|u|} f/\prod_{j\in u} \partial x_j$,
and by convention $\partial^\emptyset f$ is $f$. 
The functions we consider belong to 
\begin{align}\label{eq:itscts}
\cf = \{f:[0,1]^d\to\real\mid \text{$\partial^{1{:}d} f$ is continuous on $[0,1]^d$}\}. 
\end{align}
Continuity here allows the partial derivatives to be taken in 
any order and it allows some interchanges of order 
between differentiation and integration. 
Our space is smaller than the Sobolev spaces usually
studied, so we call it a pre-Sobolev space.
A Sobolev space also includes functions with derivatives in the sense of distributions
which then implies completeness~\cite{meye:serr:1964}.
We need the somewhat stronger condition~\eqref{eq:itscts} in order to apply
a Poincar\'e inequality.  Many applications have an integrand that satisfies~\eqref{eq:itscts}.

\subsection{Weighted spaces}

Let $\gamma_u>0$ for all $u\subseteq1{:}d$ and 
let  $\bsgamma$ comprise all of those choices for $\gamma_u$. 
We use an unanchored norm  defined by 
\begin{align}\label{eq:gamnorm}
\Vert f\Vert^2_\bsgamma 
=\sum_{u\subseteq1{:}d}\frac1{\gamma_u}
\int_{[0,1]^{|u|}}\biggl(\int_{[0,1]^{d-|u|}}\partial^u f(\bsx)\rd\bsx_{-u}
\biggr) ^2 \rd\bsx_u. 
\end{align}
See \cite{dick:kuo:sloa:2013} for background on this and a 
related anchored norm. 
This norm is finite for every choice of $\bsgamma$ and all $f\in\cf$. 
Some of our results allow a semi-norm instead found by dropping 
$u=\emptyset$ from the sum, or equivalently, taking $\gamma_\emptyset=\infty$. 
More generally, the $u=\emptyset$ term is $\mu^2/\gamma_\emptyset$
where $\mu = \int f(\bsx)\rd\bsx$. 
There are now many efficient methods of constructing 
quasi-Monte Carlo point sets for weighted spaces. 
See \cite{sloa:kuo:joe:2002}, \cite{nuyu:cool:2006}
and \cite{nuyu:cool:2006:nonp}. 

Numerous choices of weights are given in the survey \cite{dick:kuo:sloa:2013},
and a comprehensive treatment is available in~\cite{nova:wozn:II}. 
Sloan and Wo\'zniakowski \cite{sloa:wozn:1998} use product weights 
$\gamma_u=\prod_{j\in u}\gamma_j$ including $\gamma_\emptyset=1$. 
Typically $1=\gamma_1\ge\gamma_2\ge\cdots\ge\gamma_{j}\ge\gamma_{j+1}\cdots>0$. 
Hickernell \cite{hick:1996} uses weights $\gamma_u = \gamma^{|u|}$ for some $\gamma\in(0,1)$. 
Such weights are commonly called order weights. 
The more general order weights in \cite{dick2006good} 
take the form $\gamma_u = \Gamma_{|u|}$
where $\Gamma_r\ge0$ is a nonincreasing function of $r$. 
The case with $\Gamma_r=0$ for all $r\ge r_0$, known as finite-order 
weights, was studied in \cite{sloa:wang:wozn:2004}. 
However, Sloan \cite{sloa:2007} points out a danger from this choice. 
We will suppose that $\Gamma_r>0$. 
Dick et al.\ \cite{dick2006good}  also consider completely general weights $\gamma_u$
but such generality sharply raises the costs 
of using those weights to design an algorithm \cite{dick:kuo:sloa:2013}. 
Product and order weights, known as POD weights, defined by 
$\gamma_u = \Gamma_{|u|}\prod_{j\in u}\gamma_j$ have proved useful 
in QMC based algorithms for solving PDEs with random coefficients 
\cite{kuo2012quasi}. 
One useful choice has $\gamma_u = (|u|!)^\alpha\times\prod_{j\in u}j^{-\beta}$
where $0<\alpha<\beta$. Note that with this choice, $\Gamma_{|u|} = (|u|!)^\alpha$
is increasing in $|u|$. 

Higher weights are placed on the more important 
subsets and other things being equal, subsets with fewer 
components are considered more important as are subsets 
containing components with lower indices. 

We partially order subsets by $|u|$ and we use a parallel notation 
$\lceil u\rceil = \max\{j\mid j\in u\}$ to order subsets 
by their largest element. By convention, $\lceil\emptyset\rceil=0$. 
Two frequently satisfied conditions on the weights are:
\begin{align}
|u| \ge s &\implies \gamma_u \le \gamma_{1:s},\quad\text{and},
\label{eq:mincard}\\
\lceil u\rceil \ge s & \implies \gamma_u \le \gamma_{\{s\}}. \label{eq:minindex}
\end{align}
For instance, \eqref{eq:mincard} holds if the highest weighted subset of
cardinality $s$ is $1{:}s$ and the associated weight $\gamma_{1:s}$ is non-increasing in $s$.
Similarly, ~\eqref{eq:mincard} holds if the highest weighted subset with maximal element $s$
is $\{s\}$ and the associated weights are non-increasing in $s$.

Condition~\eqref{eq:minindex} 
implies that $\gamma_u \le \gamma_{\{1\}}$ for all $u\ne\emptyset$. 
This also holds for the POD weights described above. 
Regarding $1/\gamma_u$ as a multiplicative penalty factor,
the singleton $\{1\}$ is the `least penalized' index subset, though 
it may not be uniquely least penalized. 
Many but not all weights in use satisfy 
$\gamma_u\ge\gamma_v$ when $u\subseteq v$. 

Most of the widely studied weights satisfy both~\eqref{eq:mincard}
and~\eqref{eq:minindex}. 
The POD weights $\gamma_u=(|u|!)^\alpha\times\prod_{j\in u}j^{-\beta}$ 
are an exception. 
For those weights, $\gamma_{u\cup \{1\}}>\gamma_u$ whenever 
$1\not\in u$. 

\subsection{Function decompositions}

This section gives notation for two function decompositions.
The first is the ANOVA decomposition 
of $L^2\bigl([0,1]^d\bigr)$.
The second is an anchored decomposition.

The ANOVA decomposition was introduced independently 
in \cite{hoef:1948} and \cite{sobo:1969}.
The decomposition extends naturally to any mean square integrable
function of $d$ independent random inputs, and
$d=\infty$ is allowed \cite{lss}.
There is additional background in \cite{sobomat}. 
In this decomposition 
$f(\bsx) = \sum_{u\subseteq 1{:}d}f_u(\bsx)$ where $f_u$
depends on $\bsx$ only through $\bsx_u$. 
The ANOVA uniquely satisfies $\int_0^1f_u(\bsx)\rd x_j=0$
whenever $j\in u$. 
For $u\ne\emptyset$, define the variance component 
$\sigma^2_u = \int_{[0,1]^d} f_u(\bsx)^2\rd\bsx$,
and take $\sigma^2_\emptyset =0$. 
Then the variance of $f$ decomposes as 
$$\sigma^2\equiv \int_{[0,1]^d} (f(\bsx)-\mu)^2\rd\bsx 
=\sum_{u\subseteq 1{:}d}\sigma^2_u.$$

A useful alternative to the ANOVA is
the anchored decomposition, $f(\bsx) = \sum_uf^*_u(\bsx)$,
from \cite{sobo:1969}.
See \cite{kuo:sloa:wasi:wozn:2010} for a unified presentation
of this and other decompositions.
One picks an anchor point $\bsa\in[0,1]^d$
and then $f^*_u(\bsx)$ depends on $\bsx$ only through the values of
$f(\bsx_v{:}\bsa_{-v})$ for $v\subseteq u$, and if $x_j=a_j$ for any $j\in u$
then $f^*_u(\bsx)=0$.
For instance $f^*_\emptyset(\bsx)=f(\bsa)$,
while for $j\in 1{:}d$,
$f^*_{\{j\}}(\bsx) = f(\bsx_{\{j\}}{:}\bsa_{-\{j\}})-f(\bsa)$
and for $1\le j<k\le d$, 
$$f^*_{\{j,k\}}(\bsx) = f(\bsx_{\{j,k\}}{:}\bsa_{-\{j,k\}})
-f(\bsx_{\{j\}}{:}\bsa_{-\{j\}})
-f(\bsx_{\{k\}}{:}\bsa_{-\{k\}})
+f(\bsa).$$
The generalization to larger $|u|$
is given in \cite{kuo:sloa:wasi:wozn:2010}.
The anchored decomposition is especially useful for integrands
with large or infinite $d$ where any function evaluations can only
change finitely many inputs from a default value which then serves
as a natural anchor. 
The function $f^*_u$ can be evaluated through $O(2^{|u|})$ evaluations of $f$.
In some cost models, the cost to evaluate  $f(\bsx)$
increases with the number of components $x_j$ that differ from $a_j$.

\section{Literature on effective dimension}\label{sec:literature}

This section is a brief survey of the literature on effective dimension. 
Effective dimension is used to explain how QMC 
can be superior to Monte Carlo (MC) in high dimensional  integration problems.
It can be used as a post-mortem to explain why QMC did or did not
bring an improvement in a problem.  Also, other things being equal,
reducing the effective dimension using methods such as those
in \cite{acbrogla97,imai:tan:2002} is expected to improve QMC performance.
Finally, the approach from \cite{kuo2010liberating} that is now called
the multidimensional decomposition method (MDM)
uses notions of effective dimension to devise algorithms for integration
and approximation.

The notion of effective dimension seems to be as old as QMC itself. 
QMC as we know it was introduced in 1951 by Richtmyer~\cite{rich:1951} working at 
Los Alamos. 
He used what we now call Kronecker sequences to study neutron chain reactions.  
Here is what he said about effective dimension on page 13:
\begin{quotation}
The peculiarity of the integrands in question is that although $k$ may be large (e.g., $>20$) and, in fact, is really indefinite, as noted above, the effective number of 
dimensions is smaller (say, $4$ or $5$). 
(The effective number of dimensions may be defined as follows:
at point $(x_1,x_2,\ldots,x_k)$ let $s$ be the smallest integer such 
that $f(x_1,x_2,\ldots,x_s,x'_{s+1},\ldots,x'_k)$ is independent of $x'_{s+1},\ldots,x'_k$. 
Then the average of $s$ over the unit cube is the effective number of dimensions for 
the integrand $f$.)
\end{quotation}
His $k$ is the nominal dimension. Inputs after $x_s$ are completely 
ignored by $f$, although that index $s$ depends on $\bsx$. 
This notion is close to the modern notion of truncation dimension described 
below. In Richtmyer's application, one could tell for a given 
sequence $x_1,\dots,x_s$ that subsequent inputs could not make any 
difference, so his measure can be estimated by sampling. 
Incidentally, 
Richtmyer's abstract is pessimistic about the prospects for QMC to improve upon MC,
based on his numerical experience. 
We can now see that his integrands were not smooth enough for him to benefit greatly from QMC, despite their low effective dimension.

Paskov~\cite{pask:1994} used quasi-Monte Carlo points to estimate the expected 
payouts of ten tranches in a collateralized mortgage obligation. The integrands 
had a nominal dimension of $360$, corresponding to one random interest rate value 
per month in a 30 year model. 
He noticed that only one of the ten integrands seemed to use all $360$
inputs.  The least number used by any tranche was $77$. The number of inputs used 
is much like Richtmyer's definition, except that Paskov is interested in 
the number used anywhere in the cube while Richtmyer is averaging the 
number used over the cube. 

Using $d$ for nominal and $k$ for effective dimension, Paskov goes on to define 
the effective dimension to be  the smallest $k$ for which 
\begin{align}\label{eq:paskdef}
\biggl|
\int_{[0,1]^k} f(x_1,\dots,x_k,0,\dots,0) 
\rd x_1\cdots \rd x_k 
-\int_{[0,1]^d} f(\bsx)\rd\bsx 
\biggr| \le \varepsilon 
\biggl|\int_{[0,1]^d} f(\bsx)\rd\bsx\biggr|
\end{align}
holds. 
This definition allows for an $x_{k+1}$ to be used by $f$, but 
with negligible impact.
It is also a form of truncation dimension. 
He used $\varepsilon = 0.001$.  He estimated the effective dimension 
for the ten tranches, getting values between $42$ and $338$. 
These values are not especially small. 
The choice of trailing zeros inside $f$ 
in~\eqref{eq:paskdef}
appears odd given that the CMO used Gaussian random variables. 
Perhaps trailing $1/2$'s  were actually used within $[0,1]^d$, corresponding 
to trailing zeros for the resulting Gaussian variables. 

By comparing  error variances from Latin hypercube sampling 
 \cite{mcka:beck:cono:1979,stei:1987}
and plain Monte Carlo,
Caflisch, Morokoff and Owen
\cite{cafmowen} infer that about $99.96\%$ of the variance for an integrand in a CMO model 
comes from its additive contributions.  
They introduced two 
notions of effective dimension in order to explain the success of QMC 
on high dimensional problems from financial valuation. 
Let $f$ have ANOVA effects $f_u$ with variances $\sigma^2_u$
for $u\subseteq1{:}d$. 
Then $f$ has effective dimension $s$ in the truncation sense if 
$s$ is the smallest integer with 
\begin{align}\label{eq:cmotrunc}
\sum_{u:\lceil u\rceil\le s}\sigma^2_u\ge 0.99 \sigma^2.
\end{align}
Like the notions of Richtmyer and of Paskov, this quantity describes the importance of the
first $s$ compnents of $\bsx$.
It  has effective dimension $s$ in the superposition sense if  $s$ is the smallest integer with 
\begin{align}\label{eq:cmosuper}
\sum_{u:|u|\le s}\sigma^2_u\ge 0.99 \sigma^2.
\end{align}
The CMO integrand mentioned above has effective dimension $1$ in the superposition sense.

The arbitrary constant $0.99$ may be explained as follows. 
If we use plain Monte Carlo 
with $n$ observations  then our variance is 
${\sigma^2}/{n}= \sum_{u}{\sigma_u^2}/{n}$. 
If more than $99$\% of the variance comes from some subset of 
effects $f_u$, then a method such as QMC which can integrate 
them at a better rate of convergence has a possibility of 
attaining a $100$-fold improvement over MC, even if it is no 
better than MC for the other effects, for a given sample
size $n$.  Users can reasonably 
ignore a method that is say twice as fast, due to tradeoffs 
in implementation difficulty or even familiarity. 
When a $100$-fold improvement is available, it should be harder to ignore. 

By writing $f(\bsx) = \sum_{u\subseteq 1{:}d}f_u(\bsx)$ where $f_u$ depends
on $\bsx$ only through $\bsx_u$,  we can bound the QMC
error by applying the Koksma-Hlawka inequality \cite{nied92} term by term. 
The integration error is at most
\begin{align}\label{eq:decomp}
\sum_{u\subseteq 1{:}d,u\ne\emptyset} D_n^*(\bsx_{1,u},\dots,\bsx_{n,u})V_\hk(f_u)
\end{align}
where $V_\hk$ is the Hardy-Krause variation of $f_u$  and $D_n^*$ is the
star discrepancy of some $|u|$-dimensional points.  
Equation~\eqref{eq:decomp} shows how the notion of superposition dimension can 
explain the success of QMC in high dimensional problems.
It is common that QMC points have very small values of $D_n^*$ above
when $|u|$ is small.   Then all we need is for $V_\hk(f_u)$
to be small for the large $|u|$, that is, for $f$
to be dominated by its low dimensional parts.

The Hardy-Krause variations in~\eqref{eq:decomp} have no practical estimates.
The superposition dimension~\eqref{eq:cmosuper} is essentially 
using variance components $\sigma^2_u$ as a proxy.  
While $V_\hk(f_u)$ and $\sigma_u$ both measure the magnitude of $f_u$,
only $V_\hk$ captures the smoothness that QMC exploits.

The mean dimension \cite{dimdist} of a (nonconstant) function 
in the superposition sense is $\sum_u|u|\sigma^2_u/\sigma^2$. 
This quantity varies continuously, and so it can distinguish between integrands 
that have the same superposition dimension \eqref{eq:cmosuper}. 
It is easier to estimate than the effective dimension
because \cite{meandim} show that it is a simple sum of  Sobol' indices~\cite{sobo:1993}.


The above definitions of effective dimension apply to integrands taken 
one at a time.  We are also interested in notions of effective dimension 
that apply to spaces of functions. 
The first such notion was due to Hickernell~\cite{hick:1998}.
Letting $f$ be randomly drawn from a space of functions he says
that those functions are proportion $p$ (such as $p=0.99$) 
of truncation dimension $s$ if
$\sum_{\lceil u\rceil\le s} \e( \sigma_u^2(f) ) = p \e(\sigma^2(f))$
where the expectation $\e(\cdot)$ is with respect to random~$f$.
They are proportion $p$ of superposition dimension $s$ if
$\sum_{|u|\le s} \e( \sigma_u^2(f) ) = p \e(\sigma^2(f))$.
He then develops expressions for these proportions in terms
of a certain shift-invariant kernel derived from the covariance
kernel of his random functions.

Wang and Fang \cite{wang:fang:2003} defined a different notion 
of effective dimension for spaces.  They consider Korobov
spaces of periodic functions on $[0,1)^d$. Those spaces are reproducing
kernel Hilbert spaces.  Letting the kernel be $K(\bsx,\bsx')$ for
$\bsx,\bsx'\in[0,1)^d$, they define the `typical functions' in
that space to be $f_{\bsx'}(\bsx)=K(\bsx,\bsx')$
for some $\bsx'\in[0,1)^d$.  For product weights, the
effective dimension  of these typical functions,
in either the superposition or truncation senses,
is independent of $\bsx'$. They adopt that common effective dimension as the
effective dimension of the space.

The bound in~\eqref{eq:decomp} applies also for the anchored decomposition $f=\sum_uf_u^*$,
or indeed for  any of the decompositions in \cite{kuo:sloa:wasi:wozn:2010}.
The series of papers
\cite{kuo2010liberating,plaskota2011tractability,wasilkowski2012liberating}
introduced  an $\varepsilon$-superposition dimension $d(\varepsilon)$
based on the anchored decomposition.  Their MDM
is a randomized  algorithm that integrates only some of the terms in the anchored decomposition. 
The algorithm can attain root mean squared error at most $\varepsilon$
(with respect to randomization) for any $f$ in the unit ball of a weighted
Sobolev space, and it does so using only $f^*_u$
for $|u|\le d(\varepsilon)$.
Lemma 1 of \cite{plaskota2011tractability} 
shows that $d(\varepsilon) = O( \log(1/\varepsilon)
/\log(\log(1/\varepsilon)))$ as $\varepsilon\to0$ for product weights with sufficiently fast decay.
The number of anchored functions $f^*_u$ that need to be considered
was shown to be $O(\varepsilon^{-1})$ in \cite{gilbert2017small}.

Equation (10) of \cite{plaskota2011tractability} gives an expression for $d(\varepsilon)$.
Definitions of $d(\varepsilon)$ are given in 
equation (9) of \cite{kuo2010liberating} (see also (25)) and 
equation (14) of \cite{plaskota2011tractability}.
The precise definitions depend on numerous additional constants
describing error bounds for integration with respect to $\bsx_u$
and how to tradeoff between two sources of error that contribute to an error bound.

The technical report \cite{effdim-tr} is a predecessor of the present paper.  
It investigates some weighted spaces where much better performance than
Monte Carlo could be attained despite high nominal dimension.
It showed that none of the integrands involved could have large ANOVA
contributions from high dimensional parts. We postpone statements of
these conditions and results to Section~\ref{sec:effdim}.
A serious weakness in that paper was that the results only apply to periodic
functions. The present paper removes that weakness.

Kritzer, Pillichshammer and Wasilkowski 
\cite{kritzer2016very} define a truncation 
dimension counterpart to the $d(\varepsilon)$ quantity
given above, using the anchored decomposition.
In their definition, the space has 
effective dimension $s$ if the difference between an 
estimate computed using just $s$ inputs and a 
computation using all inputs is below a some small multiple $\varepsilon$ of the 
norm of a ball in that weighted space.  The bound still holds when 
the `all inputs' computation is hypothetical, as it could be 
for infinite dimensional problems. 

Similar but not quite identical definitions to the ones given here have
been used in the information based complexity literature
for the anchored decomposition.
Equation (4) of~\cite{gilbert2017small} (see also~\cite{hinrichs2018truncation}) define the set $\ca\subseteq1{:}d$
to be an active set if 
\begin{align}\label{eq:their4}
\left|
\mathcal{S}\Biggl(\sum_{u\not\in\ca(\varepsilon)}f^*_u\Biggr) 
\right|
\le \varepsilon 
\left\Vert \sum_{u\not\in\ca(\varepsilon)}
f^*_u\right\Vert_{\cf},\quad \forall f\in\cf,
\end{align}
where $\mathcal{S}$ is the operator of interest, here integration over the unit cube.
They define the $\varepsilon$-superposition dimension of $f$ to be
$$
d^{\mathrm{SPR}}(\varepsilon) = 
\min\Biggl\{
\# \ca\mid \text{$\ca$ satisfies~\eqref{eq:their4}}
\Biggr\}. 
$$
While these definitions make sense for the anchored decomposition,
they fail for the ANOVA decomposition.  Because $\int f_u(\bsx)\rd\bsx=0$
for $u\ne\emptyset$ we find that $\ca =\{\emptyset\}$ is an active set for any $\varepsilon>0$
then we always get $d^{\mathrm{SPR}}(\varepsilon) = 0$. 
Intuitively, if we could actually compute the truncated ANOVA decomposition,
we would have the exact answer from $f_\emptyset$ and get superposition or truncation dimension zero.

Recently, Sobol' and Shukhman \cite{sobol2018average} look at some particle
transport problems, similar to what Richtmyer studied.  They show that those
problems typically have low mean dimension.

Our focus here is on integration and an $L^2$ approach that allows
comparisons to Monte Carlo. Effective
dimension has also been studied for approximation problems.
See \cite{wasilkowski2014tractability,kritzer2018truncation} and references therein.
The latter reference includes a truncation dimension for approximation.
See also \cite{hinrichs2018truncation}, who  consider general linear operators. 
There has been some work outside of $L^2$, on Banach spaces \cite{kritzer2017note}.
The results we present are for unanchored spaces. Many of the papers
in information based complexity work instead with anchored spaces.
Hefter and Ritter \cite{hefter2015embeddings}
establish bounds on the norms of embeddings between 
anchored spaces and the unanchored spaces we consider here. 
Such embeddings allow error bounds in unanchored spaces
to be translated into bounds on anchored spaces, although one
then has to keep track of the norms of those embedding operators.

The literature has a third sense in which functions can have 
a low effective dimension.
L'Ecuyer and Lemieux \cite{lecu:lemi:2000} give a `successive dimensions'
sense in which the scale of non-empty  $u\subseteq1{:}d$ is given by
$1+\max\{ j\mid j\in u\}-\min\{ j\mid j\in u\}$  and $\emptyset$
has scale $0$.  A reviewer points out that effective dimension in
a generalized sense can be defined through a nested sequence of $M+1$ 
sets of subsets of $1{:}d$ given by
$\emptyset = U_0 \subset U_1 \subset U_2 \subset \cdots \subset U_M = 2^{1{:}d}$.
Then the general notion of effective dimension is
$$
\min\Bigl\{ s\mid \sum_{u\not\in U_s}\sigma^2_u(f) < \varepsilon \sigma^2(f)\Bigr\}.
$$

\section{Some Poincar\'e type inequalities}\label{sec:poincare}

Our main tool will  be bounds on $L^2$ norms based 
on integrated squared derivatives. This section presents them with 
some history.  These 
are generally known as Poincar\'e inequalities. 
Poincar\'e worked with integrals over more general domains 
than $[0,1]$ as well as more general differential operators than those used here. 

\begin{theorem}\label{thm:leastpend1}
Let $f$ be differentiable on the finite interval $(a,b)$
and satisfy $\int_a^bf(x)\rd x=0$. 
Then 
\begin{align}\label{eq:leastpend1}
\int_{a}^b f'(x)^2\rd x 
\ge 
\Bigl(\frac\pi{b-a}\Bigr)^2{\int_{a}^b f(x)^2\rd x},
\end{align}
and equality is attained for some nonzero $f$. 
\end{theorem}
\begin{proof}
This is on pages 295--296 of \cite{stek:1901}. 
\end{proof}

The constant $\pi$ in equation~\eqref{eq:leastpend1} will appear throughout 
our formulas. 
Things would be different if we were to work on the interval $[0,\pi]$ or $[0,2\pi]$. 
We retain a focus on $[0,1]$ because weighted 
spaces are almost always defined over $[0,1]^d$.

Theorem~\ref{thm:leastpend1} is commonly known as Wirtinger's 
theorem, though Stekloff's work was earlier. 
Sobol' \cite{sobo:1963} has a different proof than Stekloff,  based on the calculus of variations. 
This theorem is often given with an additional condition that $f(a)=f(b)$
(periodicity), though such a condition is not necessary. 

This inequality could be much older than 1901. 
See \cite[Chapter 2]{mitr:peca:fink:1991}. 
The condition $\int_a^bf(x)\rd x = 0$ can be removed if 
$f(a)=f(b)=0$. Equality is attained if $f(x)$ is a multiple of 
$\sin( \pi( x-m)/s)$ for $m=(a+b)/2$ and $s=b-a$.

\begin{lemma}\label{lem:multileastpen}
Let $f(\bsx)$ defined on $[0,1]^d$
satisfy $\int_0^1f(\bsx)\rd x_j=0$
for $j=1,\dots,d$. If $d\ge r\ge 0$, 
and $\partial^{1{:}r}f$ exists, then 
\begin{align}\label{eq:multileastpen}
\int_{[0,1]^d}
\Bigl( \frac{\partial^r}{\partial x_1\cdots\partial x_r}f(\bsx) 
\Bigr)^2\rd\bsx 
\ge \pi^{2r}\int_{[0,1]^d}f(\bsx)^2\rd\bsx. 
\end{align}
\end{lemma}
\begin{proof}
This holds for $r=0$ by convention. 
It holds for $d=r=1$ by 
Theorem~\ref{thm:leastpend1}. 
It extends to $d\ge r\ge1$ by induction. 
\end{proof}

\begin{theorem}\label{thm:normlowbound}
Let $f$ be defined on $[0,1]^d$ with 
$\partial^{1{:}d} f$ continuous. 
Furthermore let $f$ have ANOVA effects 
$f_v$ for $v\subseteq1{:}d$ with variance components 
$\sigma^2_v$. Then 
\begin{align}\label{eq:normlowbound}
\Vert f\Vert_\bsgamma^2 
\ge 
\gamma_\emptyset^{-1}\mu^2 
+\sum _{u\ne\emptyset}\gamma_u^{-1}\pi^{2|u|}\sigma^2_u. 
\end{align}
\end{theorem}
\begin{proof}
Let $f$ have ANOVA effects $f_v$ for $v\subseteq1{:}d$. 
Then for $u\ne\emptyset$,
\begin{align*}
&\int_{[0,1]^{d-|u|}}\partial^u \sum_{v\subseteq 1{:}d}f_v(\bsx)\rd\bsx_{-u}
=\partial^u \int_{[0,1]^{d-|u|}}\sum_{v\supseteq u}f_v(\bsx)\rd\bsx_{-u}\\
&=\partial^u\int_{[0,1]^{d-|u|}} f_u(\bsx)\rd\bsx_{-u} = \partial^u f_u(\bsx_u{:}\bszero_{-u}),
\end{align*}
because $f_u$ does not depend on $\bsx_{-u}$. 
Next 
$$\int_{[0,1]^{|u|}}(\partial^uf_u(\bsx_u{:}\bszero_{-u}))^2\rd\bsx_u 
=\int_{[0,1]^{d}}(\partial^uf_u(\bsx))^2\rd\bsx 
\ge \pi^{2|u|}\sigma^2_u$$
by Lemma~\ref{lem:multileastpen}. 
The result follows from~\eqref{eq:gamnorm}. 
\end{proof}

\section{Effective dimension of a space}\label{sec:effdim}


We begin with a ball of real-valued functions on $[0,1]^d$ given by 
$$
\cb(\bsgamma,\rho) = \{
f\in\cf \mid \Vert f\Vert_\bsgamma \le \rho\}
$$
where $\rho>0$. 
We pick the radius $\rho$ of the ball to make it just large enough 
to contain a function of unit variance.  That radius is 
\begin{align}\label{eq:defrhostar}
\rho^*=\rho^*(\bsgamma) 
= \inf\bigl\{ \rho >0\mid \exists f\in \cb(\bsgamma,\rho) 
\text{\ with $\sigma^2(f)=1$}\bigr\}. 
\end{align}
Monte Carlo sampling of a function $f\in\cb(\bsgamma,\rho^*)$
can have a root mean squared error as high as $1/\sqrt{n}$.

The ball is said to be of effective dimension $s_T$ in the truncation 
sense at level $1-\varepsilon$ (such as $\varepsilon =0.01$) 
if the sum of $\sigma^2_u$ for all sets $u$ that include
any index $j>s_T$ is below $\varepsilon$ for every $f\in \cb(\bsgamma,\rho^*)$,
and if the same is not true of $s_T-1$.  That is,
\begin{align}\label{eq:effdimspacetrunc}
s_T = \min\biggl\{ s\in\natu\mid 
\sup_{f\in\cb(\bsgamma,\rho^*)}\sum_{u:\lceil u\rceil > s}\sigma^2_u(f) < \varepsilon
\biggr\},
\end{align}
or equivalently
\begin{align}\label{eq:effdimspacetrunc2}
\sup_{f\in\cb(\bsgamma,\rho^*)}\sum_{u:\lceil u\rceil > s_T}\sigma^2_u(f) < \varepsilon  
\le \sup_{f\in\cb(\bsgamma,\rho^*)}\sum_{u:\lceil u\rceil \ge s_T}\sigma^2_u(f).  
\end{align}  
For superposition dimension, we replace the largest index by
the cardinality.
The ball $\cb(\bsgamma,\rho^*)$ is of effective dimension $s_S$ in the superposition 
sense at level $1-\varepsilon$, for 
\begin{align}\label{eq:effdimspacesuper}
s_S =\biggl\{ \min s\in\natu
\mid \sup_{f\in\cb(\bsgamma,\rho^*)}\sum_{u: |u| > s}\sigma^2_u(f) < \varepsilon\biggr\}. 
\end{align}

If the ball is of low effective dimension, then no function in it 
contains any non-negligible high order components,
where order is quantified by $|u|$ or $\lceil u\rceil$. 
Here $\varepsilon = 0.01$ corresponds to the usual choice,
but we will want to see how effective dimension depends on $\varepsilon$. 

Normalizing to variance one is interpretable, but not really necessary. 
We can work with ratios of norms. 
An equivalent condition to~\eqref{eq:effdimspacetrunc}
is that $s_T$ is the smallest integer $s$ for which 
$$
\sup_{f:\, 0<\Vert f\Vert_\bsgamma <\infty} \frac{\sum_{u:\lceil u\rceil \ge s+1}\sigma_u^2(f)}{\Vert f\Vert^2_\bsgamma} < \frac\varepsilon{{\rho^*}^2}.
$$

\begin{proposition}\label{prop:usualminrho}
Let weights $\gamma_u$ be such that 
$\gamma_{\{1\}}\pi^{-2}\ge \gamma_u \pi^{-2|u|}$
for all non-empty $u\subseteq1{:}d$. 
Then the smallest $\rho$ for which 
$\cb(\bsgamma,\rho)$ contains a function of variance $1$
is $\rho^*(\bsgamma) = \pi(\gamma_{\{1\}})^{-1/2}$. 
\end{proposition}
\begin{proof}
Let $f$ have $\sigma^2(f)=1$ and suppose that $\rho<\rho^*$. 
Then by Theorem~\ref{thm:normlowbound},
$$\Vert f\Vert_\bsgamma^2 \ge \mu^2\gamma_\emptyset^{-1}
+\min_{u\ne\emptyset}\gamma_u^{-1}\pi^{2|u|}
\ge \gamma_{\{1\}}^{-1}\pi^{2}
=(\rho^*)^2>\rho^2,
$$
so $f\not\in\cb(\bsgamma,\rho)$. 
If $f(\bsx)=\sqrt{2}\sin( \pi(x_1-1/2))$,
then $\sigma^2(f)=1$ and $f\in\cb(\bsgamma,\rho^*)$. 
\end{proof}
Weights with $\gamma_u\le\gamma_{\{1\}}$ for $u\ne\emptyset$ automatically 
satisfy the condition in Proposition~\ref{prop:usualminrho}. 

\begin{proposition}\label{prop:sigmabound}
For weights $\bsgamma$, let $\rho^*(\bsgamma)$
be defined by~\eqref{eq:defrhostar}, 
let $u\subseteq1{:}d$ be non-empty,
and let $f\in\cb(\bsgamma,\rho^*(\bsgamma))$. 
Then 
$$\sigma^2_u(f) \le \rho^*(\bsgamma)^2 \gamma_u\pi^{-2|u|}.$$
If also $\gamma_{\{1\}}\pi^{-2}\ge \gamma_u \pi^{-2|u|}$, then 
$$\sigma^2_u(f) \le 
\pi^{-2(|u|-1)}
\gamma_u /\gamma_{\{1\}}.$$
\end{proposition}
\begin{proof}
From Theorem~\ref{thm:normlowbound},
${\rho^*}^2\ge \Vert f\Vert^2_\bsgamma \ge \gamma_u^{-1}\pi^{2|u|}\sigma^2_u$,
establishing the first claim. Then the second one 
follows from Proposition~\ref{prop:usualminrho}. 
\end{proof}

\section{Bounds on effective dimension}\label{sec:bounds}

In this section we derive bounds on effective dimension
under assumption~\eqref{eq:mincard} where any set
$u$ of cardinality $s$ or larger, has $\gamma_u\le\gamma_{1:s}$.
We also consider  assumption~\eqref{eq:minindex} where any set $u$ containing 
an index $s$ or larger, has $\gamma_u\le\gamma_{\{s\}}$.

\begin{theorem}\label{thm:effdimbounds}
Let the weights $\bsgamma$ satisfy~\eqref{eq:mincard}. 
Then the corresponding pre-Sobolev space
has effective dimension in the superposition sense at level $1-\varepsilon$
no larger than
\begin{align}\label{eq:effdimboundsuper}
\max\{s \ge 1\mid \gamma_{1{:}s}  \ge \pi^{2(s-1)}\gamma_{\{1\}}\varepsilon\}. 
\end{align}
If the weights $\bsgamma$ satisfy~\eqref{eq:minindex},
then the corresponding pre-Sobolev space
has effective dimension in the truncation sense at level $1-\varepsilon$
no larger than
\begin{align}\label{eq:effdimboundtrunc}
\max\{s \ge 1\mid \gamma_{\{s\}}  \ge \gamma_{\{1\}}\varepsilon\}. 
\end{align}
\end{theorem}
\begin{proof}
Let $f\in\cb(\bsgamma,\rho^*)$ where $\rho^*$ is given by~\eqref{eq:defrhostar}. 
Let $f$ have ANOVA variance components $\sigma^2_u$. 
Choose an integer $s> 0$. 
From Proposition~\ref{prop:usualminrho},
Theorem~\ref{thm:normlowbound} and condition~\eqref{eq:mincard},
\begin{align*}
\gamma_{\{1\}}^{-1}\pi^2  
= {\rho^*}^2\ge\Vert f\Vert^2_\bsgamma  
\ge  
\mu^2\gamma_\emptyset^{-1}+\sum_{u\ne\emptyset}\gamma_u^{-1}\pi^{2|u|}\sigma^2_u  
\ge \gamma_{1:s}^{-1}\pi^{2s}\sum_{|u| \ge s}\sigma^2_u.  
\end{align*} 
Therefore 
$\sum_{|u|\ge s}\sigma^2_u \le \gamma_{1{:s}}\pi^{-2(s-1)}/\gamma_{\{1\}}$. 
If $\gamma_{1{:}s} < \gamma_{\{1\}}\pi^{2(s-1)}\varepsilon$
then $\sum_{|u|\ge s}\sigma^2_u<\varepsilon$ and the effective 
dimension cannot be as large as $s$. 
This establishes the superposition bound~\eqref{eq:effdimboundsuper}. 

For the truncation dimension, we find that 
$\gamma_{\{1\}}^{-1}\pi^2 
\ge \gamma_{\{s\}}^{-1}\pi^{2}\sum_{\lceil u\rceil \ge s}\sigma^2_u$,
by a similar argument to the one used above, this time using~\eqref{eq:minindex}. 
Then~\eqref{eq:effdimboundtrunc} follows just as~\eqref{eq:effdimboundsuper} did. 
\end{proof}

\subsection{Tractability and product weights}
Here we look at product weights 
of the form $\gamma_u = \prod_{j\in u}\gamma_j$
for monotone values $\gamma_j\ge\gamma_{j+1}$ for $j\ge1$,
including $\gamma_\emptyset=1$. 
Sloan and Wo\'zniakowski \cite{sloa:wozn:1998} 
give conditions on the weights for high dimensional quadrature to be tractable,
which we define next. 
They consider a sequence of $d$-dimensional settings 
in which $d\to\infty$. We will look at their weights 
restricted to $u\subseteq1{:}d$. 
We draw on the summary of tractability results given by~\cite{kuo:schw:sloa:2011}. 

Suppose that $f\in\cb(\bsgamma,\rho)$, for $0<\rho<\infty$. 
If we had to pick an $n=0$ point rule for functions in $\cb(\bsgamma,\rho)$
it would be a constant and we minimize worst case error 
by taking that constant to be $0$. 
Our initial error is then 
$\sup_{f\in\cb(\bsgamma,\rho)} \bigl| \int f(\bsx)\rd\bsx\bigr|.$
Now let $n=n_\bsgamma(\varepsilon,d)$ be the smallest integer for 
which some QMC rule reduces the initial error by a factor of 
$\varepsilon$. This $n$ does not depend on our choice $\rho$. 

The problem of quadrature is defined to be tractable if 
there exist points $\bsx_i$ with $n_\bsgamma(\varepsilon,d) \le Cd^q\varepsilon^{-p}$ 
for non-negative constants $C$, $p$ and $q$. 
If $q=0$ is possible, then the cost $n$ can be taken independent 
of dimension $d$, and the problem is then said to be strongly tractable. 
The problem of quadrature is strongly tractable if $\sum_{j=1}^\infty \gamma_j<\infty$
\cite{sloa:wozn:1998}. 
That result was nonconstructive and it had $p=2$, 
comparable to plain Monte Carlo. 
Hickernell and Wozniakowski
\cite{hick:wozn:2000} gave an improved non-constructive 
proof showing errors $n^{-1+\delta}$ are possible 
(so $p=1/(1-\delta)$) if 
\begin{align}\label{eq:strongertract}
\sum_{j=1}^\infty\gamma_j^{1/2}<\infty. 
\end{align}
Constructions attaining those rates have been found \cite{sloa:kuo:joe:2002,nuyu:cool:2006,nuyu:cool:2006:nonp}.

Now suppose that an error of $O(n^{-1+\delta})$ is attainable by QMC 
in a ball $\cb$ where MC would have root mean square 
error $O(n^{-1/2})$.  If that ball contains functions with significantly 
large high dimensional interactions then the QMC method must 
be successfully integrating some high dimensional functions.  If 
the ball has no functions with high dimensional components 
then the success of QMC is attributable to its performance on low 
dimensional functions.


\subsection{Effective dimension}
We consider weight factors of the 
form $\gamma_j = j^{-\eta}$ for various $\eta$. 
Some `phase transitions' happen at special values of $\eta$.
For any $\eta>2$,
\eqref{eq:strongertract} holds and then  
there exist QMC points with worst case errors  
that decrease at the rate $O(n^{-1+\delta})$ for any $\delta>0$.  
For any $\eta>1$, strong tractability holds but not at  
a better rate than Monte Carlo provides. 
These weight choices always have $\gamma_1=1$ as \cite{sloa:wozn:1998} do. 
Taking $\gamma_1>1$ would yield $\gamma_{\{1\}}>\gamma_\emptyset$
which is incompatible with the customary rule that 
smaller sets get higher weight. 

From Theorem~\ref{thm:effdimbounds} we see that 
the truncation dimension satisfies
\begin{align}\label{eq:stbound}
s_T(\varepsilon) \le \max\{ s\ge 1\mid s^{-\eta}\ge 
\varepsilon\}
\end{align}
and the superposition dimension satisfies
\begin{align}\label{eq:ssbound}
s_S(\varepsilon)\le\max\{ s\ge 1\mid (s!)^{-\eta}\ge \pi^{2(s-1)}\varepsilon\}. 
\end{align}
Both of these effective dimensions are non-increasing 
with $\eta$. 

\begin{lemma}\label{lem:dimrate}
For $\eta>0$ the dimensions given in equations~\eqref{eq:stbound}
and~\eqref{eq:ssbound} satisfy
$s_T \le \varepsilon^{-1/\eta}$
and $s_S = O( \log(1/\varepsilon)/\log(\log(1/\varepsilon)))$ as $\varepsilon\to0$.
\end{lemma}
\begin{proof}
The first claim is immediate because $s_T^{-\eta}\ge\varepsilon$.
For the second claim, $s_S$ satisfies $(s!)^{-\eta}\ge\varepsilon$.
For any $\lambda\in(0,1)$ and 
very small $\varepsilon$ we get $s$ large enough that
$s!\ge s^{\lambda s}$. Then $\varepsilon^{-1/\eta}\ge s^{\lambda s} = e^{\lambda s\log(s)}$
and so $s\log(s)\le \log(\varepsilon^{-1})/(\lambda\eta)$.
Then $w\le \log(A)/W(A)$ where $A = \log(\varepsilon^{-1})/(\lambda\eta)$
and $W$ is the principal branch of the Lambert function that solves $W(x)e^{W(x)}=x$.
To complete the proof, we note that $W(x)$ is asymptotic to $\log(x)$
as $x\to\infty$, from \cite{corless1996lambertw}.
\end{proof}

The convergence rate in Lemma~\ref{lem:dimrate}
is the same as in Lemma 1 of \cite{plaskota2011tractability} for
$d(\varepsilon)$ in MDM.
From the proof of Lemma~\ref{lem:dimrate} we find that
$$
\lim_{\varepsilon\to0} 
s_S(\varepsilon) \Bigm/
\frac{\log(1/\varepsilon)/(\eta\lambda)}{\log(\log(1/\varepsilon)/(\eta\lambda))}
\le 1
$$
for any $\lambda\in(0,1)$.


\begin{table}
\centering 
\caption{\label{tab:effdim}
Upper bounds on effective dimension for product weights 
defined by $\gamma_j=j^{-\eta}$. 
}
\begin{tabular}{lcccccc}
\toprule 
&\multicolumn{3}{c}{Truncation}&\multicolumn{3}{c}{Superposition}\\
$\varepsilon$ & $\eta=2$ & $\eta=1$ & $\eta=0$& $\eta=2$ & $\eta=1$ & $\eta=0$\\
\midrule 
$0.1$ & \phz3&\phz\phz\phz10&$\infty$ &1&1&2\\
$0.01$ & \phz9 &\phz\phz100 &$\infty$ &2&2&3\\
$0.001$ & 31 & \phz 1000 &$\infty$ &2&3&4\\
$0.0001$ & 99& 10000 &$\infty$ &3&3&5\\
\bottomrule 
\end{tabular}
\end{table}


The effective dimensions in both the truncation 
and superposition senses for product weights 
with $\eta\in\{0,1,2\}$ are given in Table~\ref{tab:effdim}. 
These values are identical to those in 
\cite{effdim-tr} under a periodicity constraint 
except that the truncation dimension for $\varepsilon=10^{-4}$
given there is $101$ instead of $100$. In this case,
equality holds in~\eqref{eq:effdimboundtrunc} for $\eta=2$ and $s=100$. 

Theorem~\ref{thm:effdimbounds} gives an upper bound on effective
dimension.  It is then possible that the bounds in Table~\ref{tab:effdim} are too high.  
In the case of the superposition dimensions there is
not much room to lower them, so the bounds in the theorem could well
be giving exact values for some $\varepsilon$.

For weights given by $\eta=2$ and using $\varepsilon = 0.01$ we find that 
the truncation dimension is at most $9$ and the superposition dimension 
is at most $2$.   When $\eta=0$, the truncation dimension is unbounded. 
For instance, an integrand of variance one depending only on $x_j$
would be inside the unit ball. Since there is no a priori upper bound on $j$
we get $\infty$ for the truncation dimension. 
It is surprising that the superposition dimension is not very large 
for $\eta=0$. For $\eta=0$, all of the weights are $\gamma_u=1$. 
Decreasing weights are commonly motivated by the reduced importance
that they give to subsets of higher cardinality, but here we see
modest truncation dimension even with constant weights.

\subsection{Most important interactions}

For $\eta=2$ and $\varepsilon =0.01$, we find that the 
only large variance components involve only 
one or two of the first $9$ input variables.  We can also investigate 
which of the two factor interactions could be large. 

Using Proposition~\ref{prop:sigmabound} we can bound $\sigma^2_u$. 
For product weights with $\gamma_j=j^{-2}$ we get 
$\sigma^2_{\{1,2\}} \le \pi^{-2}/4\doteq0.025$ and 
$\sigma^2_{\{1,3\}} \le \pi^{-2}/9\doteq 0.011$. 
The other two variable interactions have 
bounds below $0.01$ as do interactions of order three and up. 

If we lower the threshold to $\varepsilon=0.001$, then 
$\sigma^2_u$ for $u=\{1,j\}$ and $2\le j\le10$ are potentially 
this large as are those for $u=\{2,j\}$ for $j=3,4,5$, but 
$\sigma^2_u<\varepsilon$ for all $|u|\ge3$. 

For $\eta=1$, the sets $u$ where 
the upper bound on $\sigma^2_u$ is below 
$\varepsilon=0.01$ are singletons $\{1\}$ to $\{100\}$
and the same two factor interactions which meet 
the $\varepsilon=0.001$ criterion for $\eta=2$. 

These or very similar subsets were found independently 
to be important in anchored spaces by Greg Wasilkowski 
who presented them at a SAMSI workshop on QMC. 
See also \cite{gilbert2017small}.

\section{Conclusions}\label{sec:conclusions}

The effectiveness of QMC sampling on nominally high 
dimensional integrands can be explained in part by those specific integrands 
having a low effective dimension \cite{cafmowen} as measured by ANOVA. 
It is only partial because ANOVA components are not necessarily smooth 
enough for QMC to be beneficial. 
It is however common for the process by which ANOVA components are 
defined to make the low order components smooth 
\cite{griebel2010smoothing}. 

This paper considers effective dimension of weighted Sobolev spaces 
without requiring periodicity of the integrands. 
Some weighted Sobolev spaces describe families of integrands over which QMC 
has uniformly good performance. 
A ball in such a space just barely large enough 
to contain an integrand of unit variance, will contain no integrands 
with meaningfully large high dimensional or high index variance components. 
Thus algorithms for integration in these settings ought to 
focus on certain low dimensional aspects of the input space. 


It is surprising that weighted spaces 
with all $\gamma_u=1$ (which does not allow tractability) 
leads to spaces with modest superposition 
dimension.   In this sense, $\gamma_u=1$ is not a model for a situation 
where all subsets of variables are equally important.  

It would be interesting to find connections or bounds between the
superposition dimension studied here and the dimension $d(\varepsilon)$
used in MDM.  In the case of product weights, both methods have to
discount high order subsets.  The MDM literature uses a tail
quantity $\sum_{j>d}\gamma_j$.  The superposition dimension here
uses $\prod_{j=1}^{d+1}\gamma_{j}$ and the 
truncation dimension uses $\gamma_{d+1}$
It may be possible to derive bounds connecting these quantities
using embeddings.

Up to this point, we have emphasized functions of low effective dimension. 
It is important to consider what happens if the integrand $f$ at hand is not 
dominated by its low dimensional components. 
If $f$ is in the unit ball in one of the weighted Sobolev spaces, then 
a modestly large $n$ can be found which will yield an integration error 
smaller than $\varepsilon$ for this $f$ and all other integrands in that ball. 
Even if $f$ is not in that ball, $\bar f = f/\Vert f\Vert_\bsgamma$ is in 
that ball, and we can be sure of an error below $\varepsilon$ for $\bar f$. 
Using MC and QMC methods, this means we have an error 
below $\varepsilon\Vert f\Vert_\bsgamma$ for $f$. 
If $\Vert f\Vert_\bsgamma$ is very large then we need a very small $\varepsilon$
to compensate.  
A function such as $f_d = \prod_{j=1}^d(x_j-1/2)$ makes a good test case. 
For product weights $\Vert f_d\Vert=(d!)^{\eta/2}$, and so for good results 
we would need $\varepsilon$ to be comparable to $(d!)^{-\eta/2}$, which 
then requires $n$ to be a power of $d!$.




\section*{Acknowledgements}
I thank the following people for helpful discussions:
Sergei Kucherenko, Josef Dick, and Fred Hickernell. 
Greg Wasilkowski made a good suggestion about 
formulating the definition of subspace dimension that 
was incorporated into \cite{effdim-tr}. 
Frances Kuo shared an early version of~\cite{kuo:schw:sloa:2011}. 
I thank Jiangming Xiang for catching an error in a previous version of this paper.
I also would like to acknowledge SAMSI (NSF, DMS-1638521), 
whose QMC workshop came as I was completing this document. 
While there, I had additional informative conversations 
about tractability, initial error, and normalization in the Poincar\'e inequailty 
with Ian Sloan, Fred Hickernell, Greg Wasilkowski and Henryk Wo\'zniakowski. 
Finally, I thank the reviewers of this paper for many helpful comments.

\bibliographystyle{siamplain}
\bibliography{effdim}
\end{document}